\documentclass[12pt,a4paper]{amsart}

\usepackage{amssymb,amsthm}
\usepackage[usenames]{color}
\usepackage{amscd}
\usepackage{amsthm}
\usepackage{graphicx}
\usepackage{hyperref}
\usepackage{enumerate}
\usepackage[all]{xy}
\usepackage[usenames,dvipsnames]{xcolor}
\usepackage{mathrsfs}
\usepackage[margin=1.45in]{geometry}
\usepackage{cleveref}
\usepackage{tikz}
\usetikzlibrary{matrix,arrows,decorations.pathmorphing}

 \newtheorem*{corollary*}{Corollary}
 \newtheorem*{construction*}{Construction}
 \newtheorem*{definition*}{Definition}
 \newtheorem*{notation*}{Notation}
 \newtheorem*{lemma*}{Lemma}
 \newtheorem*{theorem*}{Theorem}
 \newtheorem*{remark*}{Remark}
 \newtheorem*{example*}{Example}
 \newtheorem*{conjecture*}{Conjecture}
 \newtheorem*{condition*}{Condition}
 \newtheorem*{result*}{Result}
 \newtheorem*{property*}{Property}

 \newtheorem*{cor*}{Corollary}
 \newtheorem*{const*}{Construction}
 \newtheorem*{defn*}{Definition}
 \newtheorem*{notn*}{Notation}
 \newtheorem*{lem*}{Lemma}
 \newtheorem*{thm*}{Theorem}
 \newtheorem*{rem*}{Remark}
 \newtheorem*{exm*}{Example}
 \newtheorem*{conj*}{Conjecture}

 \newtheorem{lemma}{Lemma}[subsection]
 
 \newtheorem{remark}[lemma]{Remark}
 
 \newtheorem{theorem}[lemma]{Theorem}
 \newtheorem{definition}[lemma]{Definition}
 \newtheorem{notation}[lemma]{Notation}

 \newtheorem{thm}[lemma]{Theorem}
 \newtheorem{prop}[lemma]{Proposition}
 \newtheorem{lem}[lemma]{Lemma}
 \newtheorem{defn}[lemma]{Definition}
 
 \newtheorem{cor}[lemma]{Corollary}

 \newtheorem{introtheorem}{Theorem}
 
 \crefname{introtheorem}{Theorem}{Theorems}
 \Crefname{introtheorem}{Theorem}{Theorems}
  \newtheorem{introthm}[introtheorem]{Theorem}
   \crefname{introthm}{Theorem}{Theorems}
 \Crefname{introthm}{Theorem}{Theorems}

  \crefname{introcorollary}{Corollary}{Corollaries}
 \Crefname{introcorollary}{Corollary}{Corollaries}

 \newtheorem{introcor}[introtheorem]{Corollary}
   \crefname{introcor}{Corollary}{Corollaries}
 \Crefname{introcor}{Corollary}{Corollaries}
 
   \crefname{introconjecture}{Conjectures}{Conjectures}
 \Crefname{introconjecture}{Conjecture}{Conjectures}
 
    \crefname{introconj}{Conjectures}{Conjectures}
 \Crefname{introconj}{Conjecture}{Conjectures}

     \crefname{introlem}{Lemma}{Lemmas}
 \Crefname{introlem}{Lemma}{Lemmas}
 
 \crefname{introremark}{Remark}{Remarks}
 \Crefname{introremark}{Remark}{Remarks}
 
  \crefname{introrem}{Remark}{Remarks}
 \Crefname{introrem}{Remark}{Remarks}

   \crefname{introprop}{Proposition}{Propositions}
 \Crefname{introrem}{Proposition}{Propositions}

   \crefname{introdefn}{Definition}{Definitions}
 \Crefname{introdefn}{Definition}{Definitions}
 
   \crefname{intronotn}{Notation}{Notations}
 \Crefname{intronotn}{Notation}{Notations}
 
   \crefname{introtask}{Task}{Tasks}
 \Crefname{introtask}{Task}{Tasks}
 
  \crefname{introprob}{Problem}{Problems}
 \Crefname{introprob}{Problem}{Problems}
 
   \crefname{introquestion}{Question}{Questions}
 \Crefname{introquestion}{Question}{Questions}

  \crefname{theorem}{Theorem}{Theorems}
 \Crefname{theorem}{Theorem}{Theorems}
  \crefname{thm}{Theorem}{Theorems}
 \Crefname{thm}{Theorem}{Theorems}

  \crefname{corollary}{Corollary}{Corollaries}
 \Crefname{corollary}{Corollary}{Corollaries}

   \crefname{cor}{Corollary}{Corollaries}
 \Crefname{cor}{Corollary}{Corollaries}

   \crefname{conjecture}{Conjecture}{Conjectures}
 \Crefname{conjecture}{Conjecture}{Conjectures}

    \crefname{conj}{Conjecture}{Conjectures}
 \Crefname{conj}{Conjecture}{Conjectures}

     \crefname{lem}{Lemma}{Lemmas}
 \Crefname{lem}{Lemma}{Lemmas}

      \crefname{lemma}{Lemma}{Lemmas}
 \Crefname{lemma}{Lemma}{Lemmas}

 \crefname{remark}{Remark}{Remarks}
 \Crefname{remark}{Remark}{Remarks}

  \crefname{rem}{Remark}{Remarks}
 \Crefname{rem}{Remark}{Remarks}

   \crefname{proposition}{Proposition}{Proposition}
 \Crefname{proposition}{Proposition}{Proposition}

    \crefname{prop}{Proposition}{Propositions}
 \Crefname{prop}{Proposition}{Propositions}

   \crefname{defn}{Definition}{Definitions}
 \Crefname{defn}{Definition}{Definitions}

   \crefname{notn}{Notation}{Notations}
 \Crefname{notn}{Notation}{Notations}

  \crefname{prob}{Problem}{Problems}
 \Crefname{prob}{Problem}{Problems}

   \crefname{question}{Question}{Questions}
 \Crefname{question}{Question}{Questions}

\newcommand{\alp}{\alpha}

\newcommand{\Ind}{\operatorname{Ind}}

\renewcommand{\Im}{\operatorname{Im}}

\newcommand{\HC}{\operatorname{HC}}

\newcommand{\Sym}{\operatorname{Sym}}

\newcommand{\Supp}{\operatorname{Supp}}

\newcommand{\im }{\operatorname{im}}

\newcommand{\R}{\mathbb{R}}

\newcommand{\bC}{\mathbb{C}}

\newcommand{\cE}{{\mathcal{E}}}

\newcommand{\cY}{{\mathcal{Y}}}
\newcommand{\cZ}{{\mathcal{Z}}}

\newcommand{\Sc}{\cS}

\newcommand{\Fre}{Fr\'echet }

\providecommand{\fa}{\mathfrak{a}}
\providecommand{\fn}{\mathfrak{n}}

\providecommand{\fh}{\mathfrak{h}}
\providecommand{\fk}{\mathfrak{k}}

\providecommand{\fg}{\mathfrak{g}}

\providecommand{\cC}{\mathcal{C}}
\providecommand{\cD}{\mathcal{D}}
\providecommand{\cE}{\mathcal{E}}

\providecommand{\cH}{\mathcal{H}}
\providecommand{\cI}{\mathcal{I}}

\providecommand{\cM}{\mathcal{M}}

\providecommand{\cS}{\mathcal{S}}

\providecommand{\cU}{\mathcal{U}}

\providecommand{\g}{\mathfrak{g}}

\newcommand{\simpAr}[2][r]{\ar@{}[#1]|-*[@]_{#2}}
\newcommand{\change}[1]{{{#1}}}
\newcommand{\DimaA}[1]{{{#1}}}
\newcommand{\DimaB}[1]{{{#1}}}
\newcommand{\DimaC}[1]{{{#1}}}

\newcommand{\RamiC}[1]{{{#1}}}

\newcommand{\oH}{\operatorname{H}}
\newcommand{\ot}{\leftarrow}
\newcommand{\into}{\hookrightarrow}
\newcommand{\onto}{\twoheadrightarrow}

\begin{document}

\author{Avraham Aizenbud}
\address{Faculty of Mathematics and Computer Science, Weizmann
Institute of Science, POB 26, Rehovot 76100, Israel }
\thanks{The first-named author was partially supported by ISF grant 687/13 and a Minerva foundation grant.}
\email{aizenr@gmail.com}

\author{Dmitry Gourevitch}
\email{dimagur@weizmann.ac.il}
\thanks{The second-named author was partially supported by ISF grant 756/12 and ERC Starting Grant StG 637912.}

\author{Bernhard Kr\"{o}tz}
\address{Institut f\"{u}r Mathematik, Universit\"{a}t Paderborn, Warburger
Str. 100, 33098 Paderborn, Germany}

\thanks{The third-named author was supported by ERC Advanced Investigators Grant HARG 268105.}
\email{bkroetz@gmx.de}

\author{Gang Liu}
\address{Institut \'{E}lie Cartan de Lorraine, Universit\'{e} de Lorraine, Ile du Saulcy, 57045 Metz,
France}
\email{gang.liu@univ-lorraine.fr}


\date{\today}
\title[Hausdorffness for Lie algebra homology of Schwartz spaces]{Hausdorffness for Lie algebra homology of Schwartz spaces and applications to the comparison conjecture}

\begin{abstract}
Let $H$ be a \DimaA{real algebraic} group acting equivariantly with finitely many orbits on a \DimaA{real algebraic} manifold $X$ and a \DimaA{real algebraic} bundle $\cE$ on $X$. Let $\fh$ be the Lie algebra of $H$.
Let $\Sc(X,\cE)$ be the space of
Schwartz sections of $\cE$. We prove that $\fh\Sc(X,\cE)$ is a closed subspace of $\Sc(X,\cE)$ of finite codimension.

We give an application of this result in the case when $H$ is a real spherical subgroup of a real reductive group $G$. We deduce an equivalence of two old conjectures due to Casselman: the automatic continuity and the comparison conjecture for zero homology. Namely, let $\pi$ be a Casselman-Wallach representation of $G$ and $V$ be the corresponding Harish-Chandra module. Then the natural morphism of coinvariants $V_{\fh}\to \pi_{\fh}$ is an isomorphism if and only if any linear $\fh$-invariant functional on $V$ is continuous in the topology induced from $\pi$. The latter statement is known to hold in two important special cases: if $H$ includes
 a symmetric subgroup, and if $H$ includes the nilradical of a minimal parabolic subgroup of $G$.
\end{abstract}

\maketitle

Note the erratum in the end.

\section{Introduction}\label{sec:intro}

In this  paper we start to develop the theoretical background for
comparison theorems in homological representation theory attached
to a real reductive group $G$.

\par The algebraic side of representation theory for $G$ is encoded
in the theory of Harish-Chandra modules $V$. These are modules for the
Lie algebra $\fg$ of $G$ with a compatible algebraic action of a fixed
maximal compact subgroup $K$ of $G$.   According to Casselman-Wallach
(see \cite[Chapter 11]{Wal2} or \cite{CasGlob} or, for a different approach, \cite{BK}) Harish-Chandra modules $V$ can be naturally completed to
smooth moderate growth modules $V^\infty$ for the group $G$.
Somewhat loosely speaking one might think of $V$, resp. $V^\infty$, as the regular, resp. smooth,
functions on some real algebraic variety.

\par Fix a subalgebra $\fh<\fg$. As $V\subset V^\infty$ we have natural
maps
$$ \Phi_p: \oH_p(\fh, V) \to \oH_p(\fh, V^\infty)\, .$$

If $\fh$ is a real spherical
subalgebra, then $\Phi_p$ is conjectured to be an isomorphism for all $p$.
Note that $\oH_p(\fh, V)$ is finite dimensional (see \Cref{compa}).
If $\fh$ is a maximal unipotent subalgebra then this conjecture  is
the still not fully established Casselman comparison theorem (see \cite{HT} for $G$ split). We show that for the case $p=0$ this conjecture is equivalent to a generalization of the Casselman  automatic continuity theorem. Namely, we prove the following theorem.

\begin{introthm}[See \S \ref{compa}]\label{thm:IntroCompa}
 Let $H\subset G$ be a real spherical subgroup and $\fh$ be its Lie algebra. Let $V$ be a Harish-Chandra module and $V^{\infty}$ be its smooth globalization. Then $\oH_0(\fh, V^\infty)$ is separated (Hausdorff) and finite-dimensional, and the natural homomorphism
$$\Phi_0:\oH_0(\fh, V)\to  \oH_0(\fh, V^\infty)$$
is an epimorphism. Furthermore, it is an isomorphism if and only if any $\fh$-invariant linear functional on $V$ is continuous in the topology induced from $V^\infty$.
\end{introthm}

In \S \ref{compa} we prove a generalization of this theorem that concerns $\oH_0(\fh, V\otimes \chi)$, where $\chi$ is any tempered character of $H$.

Since the automatic continuity of $\fh$-invariant functionals was shown in several cases in  \cite{BK,vdBD,BD}, we obtain the following corollary.

\begin{introcor}\label{cor:main}[See \S \ref{compa}]
If \RamiC{either }$H$ contains a symmetric subgroup of $G$ \DimaB{ fixed by an involution that commutes with the Cartan involution or} $H$ contains the unipotent radical of a minimal parabolic subgroup of $G$ then $\Phi_0$ is an isomorphism.
\end{introcor}


The main part of \Cref{thm:IntroCompa} is the statement that $\oH_0(\fh, \pi)$ is separated and finite-dimensional. By  the Casselman subrepresentation theorem it reduces to the case of principal series. This case in turn is a special case of the following statement.

\par Let  $H$ be a real algebraic group.
Let $X$ be a real algebraic manifold and $\cE$ be a real algebraic vector bundle over $X$. Assume that $H$ acts
equivariantly on $X$ and $\cE$. Let $\Sc(X,\cE)$ be
the space of Schwartz sections with respect to $\cE\longrightarrow X$.
Then $\Sc(X,\cE)$ is a  nuclear Fr\'{e}chet  space and $H$ acts smoothly on $\Sc(X,\cE)$.  Let $\fh$ be the Lie algebra of $H$.

Since the action of $H$ on $\Sc(X,\cE)$ is
smooth, the space $\Sc(X,\cE)$ becomes an $\fh$-module.
Equipped with the quotient topology, each homological space $\oH_i(\fh,\Sc(X,\cE))$ becomes a
 topological vector space.
The main theorem of this paper is:

\begin{introtheorem}[See \S \ref{sec:PfMain}]\label{thm:main} Suppose that the number of $H$-orbits in $X$ is finite and let $\chi$ be a tempered character of $H$. Then
$\oH_0(\fh,\Sc(X,\cE)\otimes \chi)$ is
separated and finite-dimensional.
In other words, the subspace $\fh\Sc(X,\cE)\subset \Sc(X,\cE)$ (where the action of $\fh$ is twisted by $\chi$) is closed and has finite codimension.
\end{introtheorem}

\DimaC{\begin{remark*}
We do not know whether the assumption on $\chi$ being tempered is necessary.
Some evidence that the assumption may be superfluous in the case of nilpotent $H$ can be found in \cite{CHM}.
\end{remark*}}

\subsection{Structure of the paper}
In \Cref{sec:prel} we give the necessary preliminaries on nuclear \Fre spaces, Schwartz functions on Nash manifolds, and the homological behaviour of inverse limits.

In \Cref{sec:PfMain} we prove \Cref{thm:main}. In order to sketch the proof let us assume that the character $\chi$ 
is trivial. We first prove it for the case of transitive action, in which we show that all the homologies of $\Sc(X,\cE)$ are finite-dimensional. We prove this using relative de-Rham complex with Schwartz coefficients, similarly to the proof of the Shapiro lemma.

To deduce the theorem, we first note that
any morphism of \Fre spaces of finite corank is a strict morphism.
We also note that the dual space to $\oH_0(\fh,\Sc(X,\cE))$, i.e. the space $\Sc^*(X,\cE)^{\fh}$ of $\fh $-invariant tempered distributional sections of $\cE$ on $X$, was shown to be finite-dimensional in \cite{AGM}, using a theorem by Bernstein and Kashiwara from the theory of D-modules. Thus, the finite-dimensionality of $\oH_0(\fh,\Sc(X,\cE))$ is equivalent to its separatedness.

We prove both statements by induction on the number of orbits. The induction step follows from the statement that for any closed orbit $Z\subset X$, the homology $\oH_0(\fh,\Sc(X,\cE)/\Sc(X\setminus Z,\cE))$ is finite-dimensional. By a generalization of E. Borel's lemma, the space $\Sc(X,\cE)$ has a decreasing filtration $F^i$ by closed subspaces, such that $F^i/F^{i+1}=\Sc(Z,\cE_i)$, for some \DimaA{real algebraic} bundles $\cE_i$ on $Z$, $\bigcap F^i=\Sc(X\setminus Z,\cE) $ and
\begin{equation}\label{=IntInvLim}
\Sc(X,\cE)/\Sc(X\setminus Z,\cE)\cong \lim_{\ot}\Sc(X,\cE)/F^i.
\end{equation}
It was shown by Grothendieck that under certain conditions homology commutes with inverse limits. In our case, these conditions are satisfied by the finiteness of the dimension of the homologies of $\Sc(Z,\cE_i)$. Thus, \eqref{=IntInvLim} implies
\begin{equation}
\oH_0(\fh,\Sc(X,\cE)/\Sc(X\setminus Z,\cE))\cong\lim_{\ot} \oH_0(\fh,\Sc(X,\cE)/F^i).
\end{equation}
Since the dimension  of $\oH_0(\fh,\lim \limits _{\ot}\Sc(X,\cE)/F^i)$ is bounded by  the dimension  of $\Sc^*(X,\cE)^{\fh}$, we get that the sequence $\oH_0(\fh,\Sc(X,\cE)/F^i)$ stabilizes and its inverse limit is finite-dimensional.

In \Cref{compa} we deduce \Cref{thm:IntroCompa} from \Cref{thm:main} and the Casselman subrepresentation theorem. We also deduce from \cite{KS2} that, under some genericity assumption (see \eqref{=gen}), all the homologies $\oH_p(\fh, V)$ are finite dimensional for any Harish-Chandra module $V$.

\subsection{Acknowledgements}

We would like to thank Eitan Sayag and Joseph Bernstein for fruitful discussions, and the anonymous referee for useful remarks.
\section{Preliminaries}
\subsection{Nuclear \Fre spaces }\label{sec:prel}

Let us begin with a brief recall on some standard facts about nuclear \Fre spaces. For more details we refer the reader to \cite[Appendix A]{CHM}.

\par A topological vector space $V$ is called {\it separated} or {\it Hausdorff} if $\{0\} \subset V$ is closed.
Non-separated topological vector spaces typically arise as quotients $V/U$ where $U\subset V$ is a non-closed
subspace of a topological vector space $V$.
\par If $V$ is a topological vector space, then we denote by $V'$ its topological dual, that is the space
of continuous linear functionals $V\to \bC$. We endow $V'$ with the strong dual topology (i.e. the topology of uniform convergence on bounded sets) and note that
$V'$ is a  separated topological vector space.

\par If $T: V \to W$ is a morphism of topological vector spaces, then we denote by $T': W'\to V'$ the corresponding
dual morphism.  A morphism $T: V\to W$ is called {\it strict}, provided that  $T$ induces an isomorphism
of topological vector spaces $V/\ker T \simeq \im T$.

In this paper we consider Nuclear \Fre spaces, NF-spaces for short. If $F$ is a NF-space and $E\subset F$ is a closed subspace then both $E$ and $F/E$ (in the induced and quotient topologies respectively) are NF-spaces. Moreover, any surjective morphism of \Fre spaces is open (see  \cite[Theorem 17.1]{Tre}). Thus, a morphism of NF-spaces is strict if and only if it has closed image. The following lemma says that morphisms of finite corank are always strict.

\begin{lemma}[{\cite[Lemma A.1]{CHM}}]\label{lem:FinGood}
Let $T:E\to F$ be a continuous linear map between $NF$-spaces. Assume that the image $\im T$ is if finite codimension in $F$. Then $T$ is a strict morphism.
\end{lemma}

In this paper we will consider homologies of modules over a Lie algebra $\fh$, as the homologies of the corresponding Koszul complexes. In particular, $\oH_0(\fh,V)=V_{\fh}:=V/\fh V$. If $V$ is a topological vector space then we consider the quotient topology on this space. The main results of this paper say that in some important cases this topology is separated.

\subsection{Schwartz spaces on Nash manifolds}\label{sec:Nash}

Nash manifolds are smooth semi-algebraic manifolds.
Nash manifolds are equipped with the \emph{restricted topology},
in which open sets are open semi-algebraic sets. This is not a
topology in the usual sense of the word as infinite unions of open
sets are not necessarily open sets in the restricted topology.
However, finite unions of open sets are open
 and therefore in the restricted topology we consider only
finite covers. In particular, if $\cE$ over $X$ is a Nash vector bundle
it means that there exists a \underline{finite} open cover $U_i$
of $X$ such that $\cE|_{U_i}$ is trivial.
A Lie group $G$ is called a {\it Nash group} provided that $G$ is a Nash manifold
and all group operations being Nash maps.
For more details  on Nash
manifolds we refer the reader to \cite{BCR,Shi} and for details  on Nash groups to \cite{Sun}.

Schwartz functions are functions that decay, together with all
their derivatives, faster than any polynomial. On $\R^n$ it is the
usual notion of Schwartz function. We also need the notion of tempered functions, i.e. smooth
functions that grow not faster than a polynomial,  and so do all their derivatives.
 For precise
definitions of those notions we refer the reader to \cite{AGSc}.
In this section we summarize some elements of the theory of Schwartz functions.

\par Fix a Nash manifold $X$ and a Nash bundle $\cE$ over $X$. We denote
by $\Sc(X)$ the \Fre space of
Schwartz functions on $X$ and by $\Sc(X,\cE)$
the space of Schwartz sections of $\cE$.
We collect a few central facts:

\begin{thm}\label{thm:ScBasic}
\begin{enumerate}
\item $\Sc(\R ^n)$ = Classical
Schwartz functions on $\R ^n$. See \cite[Theorem 4.1.3]{AGSc}.
\item \label{p:SchFre} The space $\Sc(X,\cE)$ is a nuclear \Fre space. See \cite[Corollary 2.6.2]{AGRhamShap}.
\item \label{prop:ExtClose} Let $Z \subset X$ be a closed Nash submanifold. Then the
restriction maps  $\Sc(X,\cE)$ onto $\Sc(Z,\cE|_Z)$. See \cite[\S 1.5]{AGSc}.
\item \label{p:Tempered} If $\alp$ is a tempered function on $X$ then $\alp\Sc(X,\cE)\subset \Sc(X,\cE)$. See \cite[Proposition 4.2.1]{AGSc}. One can show that this is a defining property of tempered functions.
\end{enumerate}
\end{thm}

\begin{prop}[{\cite[ Theorem 5.4.3]{AGSc}}] \label{pOpenSet}
Let $U \subset X$  be a (semi-algebraic) open subset, then
$$\Sc(U,\cE) \cong \{\phi \in \Sc(X,\cE)| \quad \phi \text{ is 0 on } X
\setminus U \text{ with all derivatives} \}.$$
In particular, extension by zero defines a closed imbedding $\Sc(U,\cE) \into \Sc(X,\cE)$.
\end{prop}

\begin{prop}[{Partition of unity,  \cite[\S 5]{AGSc}}] \label{pCosheaf} Let $X = \bigcup_{i=1}^n U_i$ be a finite open
cover of $X$.  Then there exists a tempered partition of unity
\DimaC{, i.e. tempered functions $\{\lambda_i\}_{i=1}^n$ on $X$ such that $1 =\sum
 _{i=1}^n \lambda_i$ and $\Supp(\lambda_i)\subset U_i$.

 Note that \Cref{pOpenSet} implies that
 for any Nash bundle $\cE$ over $X$} and any Schwartz section $f\in \Sc(X,\cE)$, the section $\lambda_i f$ is a Schwartz section of $\cE$ on $U_i$ (extended by zero to $X$).
\end{prop}

\DimaC{Let $Z$ be a closed semi-algebraic subset of  $X$.  Denote
$$\Sc_Z(X,\cE):=\Sc(X,\cE)/\Sc(X-Z,\cE).$$
Here we identify $\Sc(X-Z,\cE)$ with a closed subspace of $\Sc(X,\cE)$
using the description of Schwartz functions on an open set (Proposition \ref{pOpenSet}).
Note that this notation differs from the notation in \cite{AGSc,AGRhamShap,AGST}, but matches the notation common in sheaf theory.}

\par To obtain a feeling for the objects $\Sc_Z(X, \cE)$ let us
 consider the case of the trivial bundle and $Z=\{pt\}$ a point.
\change{Then $\Sc_{\{pt\}}(X)= \Sc(X) / \Sc(X-\{pt\})$ and Proposition \ref {pOpenSet} implies that
there is a well defined
injective map (the Taylor series map at the point $ pt$) into the ring of power series in $n=\dim X$ variables:

$$ \Sc_{\{pt\}} (X) \to \bC[[x_1, \ldots, x_n]]\,  .$$}
The contents of Borel's Lemma is that this map is surjective.  Note that the formal power  series
have a natural structure as projective limit. The generalization of Borel's lemma now reads as:

\DimaC{\begin{lem}[{See \cite[Lemmas B.0.8 and B.0.9]{AGST}}] \label{lem:Borel}
Let $Z \subset X$ be a closed Nash submanifold.

Then $\Sc_Z(X,\cE)$ has a canonical countable decreasing filtration by closed subspaces
$\Sc_Z(X,\cE)^i$ satisfying:
\begin{enumerate}
\item
$\bigcap \Sc_Z(X,\cE)^i=\{0\}$.

\item $gr_i(\Sc_Z(X,\cE)) \cong \Sc(Z,\Sym^i(CN_Z^X)\otimes \cE)$, where $CN_Z^X= (TX|_Z/ TZ)^*$
denotes the conormal bundle to $Z$ in $X$.

\item The natural map  $$\Sc_Z(X,\cE) \to \lim_\ot(\Sc_Z(X,\cE)/\Sc_Z(X,\cE)^i)$$
is an isomorphism.
\end{enumerate}
\end{lem}}

\change{The following universal finiteness bound is a key technical tool for this paper:}

\begin{theorem}[{\cite[Theorem D]{AGM}}]\label{thm:Fin}
Let a real algebraic group $H$ act on a real algebraic manifold $X$ with
finitely many orbits. Let $\fh$ be the Lie algebra of $H$. Let $\cE$ be an algebraic $H$-equivariant
bundle on $X$. Then, there exists $C\in \mathbb{N}$ such that for every
character $\chi$ of $\fh$ we have
$$\dim \Sc^*(X,\cE)^{(\fh,\chi)}<C.$$
\end{theorem}

We now recall some facts from de-Rham theory with Schwartz coefficients. For more details see \cite[\S 3]{AGRhamShap}.

\begin{notation}\label{twist}[cf. {\cite[Definition 3.2.1]{AGRhamShap}}]
Let $\pi:F\to X$ be (Nash) locally trivial fibration of Nash manifolds, and $\cE$ a (Nash)  vector bundle on $X$. We denote by $TDR_{\Sc}^{\cE}(F \to X)$ the twisted relative de-Rham complex with coefficients in Schwartz sections of $\cE$, i.e. the complex whose entries are $\Sc(F,\pi^*(\cE)\otimes \Omega^k_{F/X}\otimes Orient_{F/X})$, where  $\Omega^k_{F/X}$ is the bundle of relative differential forms and $Orient_{F/X}$ is the relative orientation bundle.

If $X$ is a point and $E$ is one-dimensional we use  $TDR_{\Sc}(F)$ instead of $TDR_{\Sc}^{\cE}(F \to X)$.
\end{notation}

The following version of the de-Rham theorem follows from \cite[Theorem 3.2.2]{AGRhamShap} \change{and Poincare} duality.

\begin{theorem}\label{thm:RelDR}
In the notations above, the cohomologies of  $TDR_{\Sc}(F\to X)$ are $\Sc(X,\cE \otimes \cH^i_{F\to X})$ where $\cH^i_{F\to X}$ are the bundles on $X$ such that for $x \in X$ the fiber $\cH^i_{F\to X}|_x$ is canonically isomorphic to \RamiC{the singular homologies} $\oH_{\dim(\pi^{-1}(x))-i}(\pi^{-1}(x)\RamiC{,\R}).$
\end{theorem}
\begin{theorem}[see {\cite[Theorem 2.4.15]{AGRhamShap}}]\label{thm:HomFin}
Let $M$ be a Nash manifold. Then all the homology groups
$\oH_i(M)$
are finite dimensional.
\end{theorem}

\begin{lem}[see {\cite[Lemma A.0.7]{AGS2}}]\label{lem:DRK}
Let $H$ be a Nash group. Then the de-Rham complex $DR_{\Sc}(H)$ is isomorphic to the Koszul complex that computes the homologies of $\Sc(H)$.
\end{lem}

\subsection{Inverse limit and Lie algebra homologies}
In this paper we only consider inverse limits sequences $\{V_n\}$ parameterized by non-negative integers, with maps $\varphi_n:V_{n+1}\to V_n$.
The inverse limit functor is left-exact on any abelian category that admits countable direct products,  but in general not exact. We will now define a well-known class of acyclic sequences, called Mittag-Leffler sequences.

\begin{definition}
Let $\dots\to V_{j+1}\to V_j\to \dots \to V_0$ be a sequence of objects of some abelian category. Let $\varphi_{ij}$ denote the map $V_i\to V_j$ obtained by composition.
We say that this is an \emph{ML sequence} if for any fixed $j$, the images of $\varphi_{ij}$ stabilize, i.e. there exists $i_0$ such that for any $i>i_0$, we have $\Im \varphi_{ij}=\Im \varphi_{i_0j}$.
\end{definition}

Clearly, if all the maps $\varphi_n$ are epimorphisms then $\{V_n\}$ is an  ML sequence. Another important example of ML sequence is any sequence in the category of finite-dimensional vector spaces.

We will need the following lemma.
\begin{lemma}[{\cite[Chapter 0, Proposition 13.2.3]{Gro}}]\label{lem:Gro}
Let $$\dots\to K^{\bullet} _{j+1}\to K^{\bullet} _{j}\to \dots \to K^{\bullet} _{0}$$ be a  sequence of complexes of abelian groups. Suppose that for any $n$, the sequence $$\dots\to K^{n} _{j+1}\to K^{n} _{j}\to \dots \to K^{n} _{0}$$ is ML, and that for some $p$, the sequence of cohomologies
$$\dots\to \oH^{p-1} _{j+1}\to \oH^{p-1} _{j}\to \dots \to \oH^{p-1} _{0}$$ is ML. Then the natural map
\change{ $$\oH^{p}(\lim_{\ot}K_j^\bullet)\to \lim_{\ot}\oH^{p}(K_j^\bullet).$$}
is an isomorphism. \end{lemma}

Applying this to the Koszul complexes of $\fh$-modules we obtain the following corollary.

\begin{cor}\label{lem:LimHom}
Let $\dots\onto V_{j+1}\onto V_j\onto\dots \onto V_0 $ be a sequence of $\fh$-modules and epimorphisms between them. Suppose that for some $p$ and any $j$, $\oH_p(\fh,V_j)$ is finite-dimensional. Then the natural map $$\oH_{p-1}(\fh,\lim_{\ot}V_j)\to \lim_{\ot}\oH_{p-1}(\fh,V_j)$$
is an isomorphism.
\end{cor}
\begin{proof}
For any $j$, let $K^{\bullet} _{j}$ denote the Koszul complex of $V_j$, parameterized in the opposite direction, in order to make the indexing compatible with the previous lemma. Since tensoring with a finite-dimensional vector space is an exact functor, all the maps in the sequence $$\dots\to K^{n} _{j+1}\to K^{n} _{j}\to \dots \to K^{n} _{0}$$ are onto for any $n$, and thus this sequence is ML. Since the homologies $\oH_p(\fh,V_j)$ are finite-dimensional, they also form a ML sequence. Thus the statement follows from \Cref{lem:Gro}.
\end{proof}

\section{Proof of \Cref{thm:main}}\label{sec:PfMain}

\change{Let $H$ be a Nash group.
A (continuous)  character $\chi: H \to \bC^{\times}$ is called tempered provided $\chi$ is
tempered as a function on $H$.

\begin{remark} \label{char}(a) All unitary characters are tempered. If $H$ is a unipotent group, then a character $\chi$ is tempered
if and only if $\chi$ is unitary. If $H$ is a reductive algebraic group, then all characters are tempered.
\par\noindent (b) If $\chi$ is a tempered character, then the self map $ f\mapsto \chi f$ on $\Sc(H)$ is an
isomorphism. This follows from \Cref{thm:ScBasic}\eqref{p:Tempered}.
\end{remark}

In the sequel all characters to be considered are required to be tempered.  Further  we use the same symbol
$\chi$ for the derived infinitesimal character $d\chi: \fh\to\bC$ of the Lie algebra $\fh$ of $H$. All
occurring infinitesimal characters of $\fh$ in this paper are requested to exponentiate to a tempered character
of $H$.}

\subsection{Case of transitive action}
The goal of  this section is to  prove the following proposition.
\begin{prop}\label{prop:Shap}
Let $H$ be a Nash group and $Z$ be a transitive Nash $H$-manifold. Let $\cE$ be a Nash $H$-equivariant bundle over $Z$ and $\chi$ be a tempered character of $H$.
Then all the homologies $\oH_p(\fh,\Sc(Z,\cE)\otimes \chi)$ are finite-dimensional.
\end{prop}

\begin{definition}
We call an $\fh$-module $V$ \emph{homologically finite} if $\oH_p(\fh,V)$ is finite-dimensional for any $p$ and \emph{$i$-homologically finite} if $\oH_p(\fh,V)$ is finite-dimensional for any $p\leq i$.
\end{definition}

\begin{lem}\label{lem:ResolFin}
Let $$\cC: \,\dots\to0\to C_{n-1}\overset{d_{n-2}}{\to} \dots\overset{d_0}{\to} C_0\to 0\to\cdots$$ be a bounded complex consisting of homologically finite $\fh$-modules. Fix $i\geq 0$ and assume that $\oH_j(\cC)$ is $i$-homologically finite for any $j> 0$.  Then $\oH_0(\cC)$ is $i+2$-homologically finite.
\end{lem}
\begin{proof}
We prove the statement by induction on the length $n$ of $\cC$. Let $k\leq i+2$. We need to show that $\oH_{k}(\fh,\oH_0(\cC))$ is finite-dimensional. Denote $I_p:=\Im d_p$ for all $p$.   The short exact  sequence
$$ 0 \to I_0\to C_0\to \oH_0(\cC)\to 0$$
gives rise  to the long exact sequence
$$ \cdots \to \oH_k(\fh,I_0)\to \oH_k(\fh,C_0)\to \oH_k(\fh,\oH_0(\cC))\to \oH_{k-1}(\fh,I_0)\to \cdots$$
We know that $\oH_k(\fh,C_0)$ is finite-dimensional and thus it is enough to show that $\oH_{k-1}(\fh,I_0)$ is finite-dimensional.   The short exact  sequence
$$ 0 \to \oH_1(\cC)\to C_1/I_1\to I_{0}\to 0$$
gives rise  to the long exact sequence
$$ \cdots \to \oH_{k-1}(\fh,\oH_1(\cC))\to \oH_{k-1}(\fh,C_1/I_1)\to \oH_{k-1}(\fh,I_0)\to \oH_{k-2}(\fh,\oH_1(\cC))\to \cdots$$
We know that $\oH_{k-2}(\fh,\oH_1(\cC))$ is finite-dimensional and thus it is enough to show that $\oH_{k-1}(\fh,C_1/I_1)$ is finite-dimensional. Let $D_p:=C_{p+1}$ and consider the complex
$$\cD: \,\dots\to0\to D_{n-2}\to \dots\to D_0\to 0\to\cdots.$$
Note that $\oH_p(\cD)=\oH_{p-1}(\cC)$ for any $p>0$ and $\oH_0(\cD)=C_1/I_{1}$. By the induction hypothesis $\oH_{k-1}(\fh,\oH_0(\cD))$ is finite-dimensional. Since $\oH_{k-1}(\fh,C_1/I_1)=\oH_{k-1}(\fh,\oH_0(\cD))$ the lemma follows.
\end{proof}

\begin{proof}[Proof of \Cref{prop:Shap}]$\,$
\begin{enumerate}[{Case} 1.]
\item  $Z=H; \, \cE $ and $\chi$ are trivial.\\
By \Cref{lem:DRK}, the Koszul complex of $\Sc(H)$ is the de-Rham complex with Schwartz coefficients of $H$, and thus by \Cref{thm:RelDR} the homologies of $\Sc(H)$ equal the homologies of $H$ as a topological space, that are finite by \Cref{thm:HomFin}.

\item \change{$Z=H; \, \cE $ and $\chi$ are arbitrary. \\As all $H$-equivariant vector bundles on $H$ are trivial we obtain with
Remark \ref{char}(b) that  $\Sc(H,\cE)\otimes \chi \simeq \Sc(H)\otimes V$ for some finite-dimensional vector space $V$ with trivial action of $H$.}

\item General case.\\
\change{Let $Z=H/L$ with $L<H$ a Nash subgroup. With $F=H$ this yields a locally trivial fiber bundle
$F\to Z, h\mapsto hL$ with orientable fibers isomorphic to $L$.
\RamiC{Recall that $TDR_{\Sc}^{\cE}(F\to Z)$ denotes} the twisted relative de-Rham complex of $F$ over $Z$ with Schwartz coefficients
(see Notation \ref{twist}).} \RamiC{Let $\cH^i_{F\to Z}$ be the bundles on $Z$ given by \Cref{thm:RelDR}. Note that $\cH^0_{F\to Z}$ is a line bundle and let $\cE'=\cE\otimes (\cH^0_{F\to Z})^*.$}

We will prove by induction on $i$ that $\Sc(Z,\cE)\otimes \chi$ is $i$-homologically finite for any bundle $\cE$ and character $\chi$. As the base we take $i=-1$.
Let $\cC$ be the chain complex given by $$\cC_k:=TDR_{\Sc}^{\cE'}(F\to Z)^{\dim F -\dim Z-k}.$$
 By \Cref{thm:RelDR} the homologies of $\cC$ are $\Sc(Z,\cE_i)$, for certain Nash bundles $\cE_i$, where $\cE_0=\cE$.
Thus the homologies of $\cC\otimes \chi$ are $\Sc(Z,\cE_i)\otimes \chi$. Hence the induction hypothesis implies that for any $j$,  $\oH_j(\cC\otimes \chi)$ is $i-1$-homologically finite. By the previous case the complex $\cC\otimes \chi$ consists of homologically finite
$\fh$-modules.
Therefore, \Cref{lem:ResolFin} implies that $\Sc(Z,\cE)\otimes \chi$, which equals $\oH_0(\cC\otimes \chi)$, is $i$-homologically finite.
\end{enumerate}
\end{proof}

\subsection{General case}
\DimaC{
\begin{lem}\label{cor:Fin}
Let a real algebraic group $H$ act on a real algebraic manifold $X$ with
finitely many orbits. Let $\fh$ be the Lie algebra of $H$. Let $\cE$ be an algebraic $H$-equivariant
bundle on $X$. Let $Z\subset X$ be a closed $G$-orbit. Consider the filtration $\Sc_Z(X,\cE)^i$ of $\Sc_Z(X,\cE)$  as in Lemma \ref{lem:Borel}.
Then there exists $C\in \mathbb{N}$ such that for every
character $\chi$ of $\fh$ and any $i$ we have
\change{$$\dim \oH_0(\fh, \big(\Sc_Z(X,\cE)/\Sc_Z(X,\cE)^i\big)\otimes \chi)<C.$$}
\end{lem}
\begin{proof}  \change{Fix $i$ and set $F_i:=\Sc_Z(X,\cE)/\Sc_Z(X,\cE)^i$. According to Lemma
\ref{lem:Borel}, $\Sc_Z(X,\cE)/\Sc_Z(X,\cE)^i$ has a finite decreasing filtration $F_i^j:= \Sc_Z(X,\cE)^j/\Sc_Z(X,\cE)^i$,
$0\leq j\leq i$, with quotients isomorphic to $\Sc(Z,\Sym^j(CN_Z^X)\otimes \cE)$}. By \Cref{prop:Shap},
$\oH_0(\fh, \Sc(Z,\Sym^i(CN_Z^X)\otimes \cE)\otimes \chi)$ is finite-dimensional.
This implies that $\oH_0(\fh, \Sc_Z(X,\cE)/\Sc_Z(X,\cE)^i\otimes \chi)$ is finite-dimensional. By \cref{lem:FinGood} this homology space is separated. Thus, it has the same dimension as its dual space, which in turn naturally embeds into the space $\Sc^*(X,\cE)^{(\fh,-\chi)}$, whose dimension is bounded by \Cref{thm:Fin}.
\end{proof}

\begin{lem}\label{lem:key2}
Let $Z\subset X$ be a closed $H$-orbit and $U$ be the complement to $Z$ in $X$. Then $(\Sc(X,E)/\Sc(U,E))_{(\fh,\chi)}$ is finite-dimensional. \end{lem}

\begin{proof} Let $\Sc_Z(X,\cE)^i\subset \Sc_Z(X,\cE)=\Sc(X,E)/\Sc(U,E)$ be as in \Cref{lem:Borel}.
By \Cref{lem:Borel}, $\Sc(X,E)/\Sc(U,E)=\lim\limits_\ot \Sc(X,E)/\Sc_Z(X,\cE)^i$. By  \Cref{lem:Borel,prop:Shap}, $\oH_p(\fh, (\Sc(X,E)/\Sc_Z(X,\cE)^i)\otimes \chi)$ are finite-dimensional, and by \Cref{lem:LimHom},  $$\oH_0(\fh, \lim\limits_\ot \Sc(X,E)/\Sc_Z(X,\cE)^i\otimes \chi)\simeq\lim\limits_\ot \oH_0(\fh, \Sc(X,E)/\Sc_Z(X,\cE)^i\otimes \chi).$$
By \Cref{cor:Fin}, the dimensions of $\oH_0(\fh, \Sc(X,E)/\Sc_Z(X,\cE)^i\otimes \chi)$ stabilize. Since for any $i$, the natural map
$$\oH_0(\fh, \Sc(X,E)/\Sc_Z(X,\cE)^{i+1}\otimes \chi)\to \oH_0(\fh, \Sc(X,E)/\Sc_Z(X,\cE)^i\otimes \chi)$$ is onto, the sequence $\oH_0(\fh, \Sc(X,E)/\Sc_Z(X,\cE)^i\otimes \chi)$ stabilizes and thus its limit is finite-dimensional.
\end{proof}
}

\begin{proof}[Proof of  \Cref{thm:main}]
We prove the theorem by induction on the number of orbits. If this number is zero the statement is trivial. Otherwise, let $Z$ denote a closed orbit and let $U$ denote its complement. Then we have the short exact sequence
$$0\to \Sc(U,E)\to \Sc(X,E)\to \Sc(X,E)/\Sc(U,E)\to 0$$
and the right-exact sequence
$$ \oH_0(\fh, \Sc(U,E)\otimes \chi)\to \oH_0(\fh, \Sc(X,E)\otimes \chi)\to \oH_0(\fh, \Sc(X,E)/\Sc(U,E))\otimes \chi)\to 0.$$
The space $ \oH_0(\fh, \Sc(U,E)\otimes \chi)$ is finite-dimensional by the induction hypothesis and the space $\oH_0(\fh,\Sc(X,E)/\Sc(U,E))\otimes \chi)$ is finite-dimensional by Lemma \ref{lem:key2}. Thus, the space $\oH_0(\fh,\Sc(X,E)\otimes \chi)$ is also finite-dimensional. It is then separated by Lemma \ref{lem:FinGood}.
\end{proof}

\section{Relation to Comparison Theorems}\label{compa}
\setcounter{lemma}{0}
In this section we let $G$ be an algebraic real reductive group.
We fix a maximal compact subgroup $K$.  We let $\cM(G)$ denote the category of smooth admissible (finitely generated) \Fre representations of moderate growth and $\cM(\fg,K)$ denote  the category of admissible (finitely generated) $(\fg,K)$-modules. For $\pi\in \cM(G)$ let $\pi^{HC}\in \cM(\fg,K)$ denote the space of $K$-finite vectors. By the celebrated result of Casselman and Wallach (see \cite{CasGlob,Wal2,BK}) the functor $\pi \mapsto \pi^{HC}$ is an equivalence of categories. The aim of this section is to study the relations between this equivalence of categories and restrictions to subalgebras of $\g$.

An algebraic subgroup $H<G$ is called {\it real spherical} if the
action of  a minimal
parabolic subgroup $P$ on $G/H$ has  an open orbit.
We recall that real sphericity implies that the double coset space $P\backslash G / H$ is
finite \cite{KS1}. Typical examples for real spherical subgroups are the nilradicals of minimal parabolic subgroups, and symmetric subgroups.

\par Note that the inclusion mapping $\pi^{HC}\into \pi$ yields a morphism in homology
$\oH_\bullet(\fh, \pi^{HC}) \to \oH_\bullet(\fh, \pi)$. For instance if
$H$ is the unipotent radical of a minimal parabolic subgroup, then the  Casselman comparison theorem (see \cite{HT} for the case of split $G$) asserts
that these two homology theories coincide. In this section we generalize some aspects of this theorem.

\subsection{Comparison for the zero homology}
\begin{thm}\label{thm:compa} Let $H\subset G$ be a real spherical subgroup,  and $\chi$ be a character of $H$.
Let $\pi\in \cM(G)$. Then
\begin{enumerate}[(i)]
\item
The space $\pi_{(\fh,\chi)}:=\oH_0(\fh, \pi\otimes \chi)$ is separated and finite-dimensional.
\item The natural homomorphism
$$\Phi_0:\oH_0(\fh, \pi^{\HC}\otimes \chi)\to  \oH_0(\fh, \pi\otimes \chi)$$
is an epimorphism.
\item $\Phi_0$ is an isomorphism if and only if the natural inclusion
$$i_0:\oH^0(\fh, \pi'\otimes(-\chi))\hookrightarrow  \oH^0(\fh, (\pi^{\HC})^{*}\otimes(-\chi))$$
is an isomorphism, where $(\pi^{\HC})^{*}$ denotes the linear dual.
\end{enumerate}
\end{thm}

\begin{remark}\label{rem:compa}
To say that $i_0$ is onto is the same as to say that any $(\fh,-\chi)$-equivariant linear functional on $\pi^{HC}$ is continuous in the topology induced from $\pi$. This property is called automatic continuity. It is clear that if it holds for $(\fh,\chi)$ then it holds for any $(\fh',\chi')$ with $\fh\subset \fh', \, \chi'|_\fh=\chi$. Automatic continuity is known to hold if $\chi$ is trivial  and $\fh$ is either a symmetric subalgebra of $\fg$ \DimaB{fixed by an involution that commutes with the Cartan involution} (by  \cite[Theorem 2.1]{vdBD} and \cite[Theorem 1]{BD}) or  the Lie algebra of the nilradical of a minimal parabolic subgroup (by \cite[Theorem 11.4]{BK}).
\end{remark}

\Cref{thm:compa,rem:compa} imply \Cref{thm:IntroCompa,cor:main}.

\begin{remark}
Twisted automatic continuity does not hold for general spherical pairs. For example, if $H$ is the nilradical of a Borel subgroup of $G$, $\chi$ is a non-degenerate unitary character of $H$ and $\pi$ is a principal series representation then $\oH^0(\fh, \pi'\otimes(-\chi))$ is one-dimensional, while the dimension of $\oH^0(\fh, (\pi^{\HC})^{*}\otimes(-\chi))$ equals the size of the Weyl group, see \cite{Kos}.
\change{However, we do expect the twisted automatic continuity to hold provided $\fh$ is self-normalising.}
\end{remark}

For the proof of \Cref{thm:compa} we will need the following dual version of the Casselman subrepresentation theorem (cf. {\cite[Theorem 8.21]{CM}}).

\begin{thm}\label{thm:CasSubRep}
Let $\pi\in \cM(G)$ be irreducible and let $P$ be a minimal parabolic subgroup of $G$. Then there exists an irreducible
finite-dimensional
representation $\sigma$ of $P$ and an epimorphism $\Ind_P^G(\sigma)\onto \pi$.
\end{thm}

\DimaB{Here, $\Ind_P^G(\sigma)$ denotes the non-normalized parabolic induction, i.e. the space of smooth functions from $G$ to $\sigma$ satisfying
$f(px)=\sigma(p)f(x).$
 }

\begin{cor}\label{cor:CasQuot}
Any irreducible $\pi\in \cM(G)$ is a quotient of a principal series representation, i.e. $\pi=\Ind_P^G(V\otimes \psi)$, where $P$ is a minimal parabolic subalgebra of $G$, $V$ is a finite-dimensional  algebraic representation of the reductive quotient $M$ of $P$, and $\psi$ is a character of $M$.
\end{cor}

\RamiC{We will also need the following definition and \DimaB{well-known} lemma.
\begin{defn}
A finite-dimensional \DimaB{complex} representation of \DimaB{a real} algebraic group is called tempered if all its matrix coefficients are tempered functions.
\end{defn}

\DimaC{Note that a finite-dimensional algebraic representation is always tempered, and the tensor product of tempered representations is tempered.
In particular, the representation $V\otimes \psi$ in \Cref{cor:CasQuot} is tempered.
 Note also that the modular character of any linear real algebraic group is always tempered, since it factors through the reductive quotient.}

\begin{lem}\label{lem:Int}
Let $(\pi,V)$ be a finite-dimensional tempered  representation of a parabolic subgroup $P\subset G$. Let $  \Delta_P$ denote the modular character of $P$. Let $G\times P$ act on $\Sc(G,V)$ by $$((g,p)f)(x):=\Delta_P(p)\pi(p)f(p^{-1}xg).$$ Then we have an epimorphism of $G$ representations
$$\varphi:\Sc(G,V)_{0\times\mathfrak{p}}\onto\Ind_P^G(V).$$
\end{lem}

\begin{proof}
Fix a right-invariant Haar measure $d\mu$ on $P$. Let $f\in \Sc(G,V)$. Define $$\bar\varphi(f)(x):=\int_{P}\pi(p^{-1})\cdot f(px)d\mu(p).$$
Note that the integral converges, since $f$ is a Schwartz function and $V$ is tempered.

We have to prove that
\begin{enumerate}
\item \label{it:barInd} $\bar\varphi(f)\in \Ind_P^G(V)$, i.e. $\bar\varphi(f)(px)=\pi(p)\bar\varphi(f)(x)$.
\item For every $\alp \in \mathfrak{p}, \bar\varphi(\alp f)=0$.
\item \label{it:barG} $\bar\varphi$ is $G$-equivariant, i.e. $(\bar\varphi(f))(xg)=(\bar\varphi(h))(x),$ where $h(x)=f(xg)$.
\item $\Im(\bar \varphi)=\Ind_P^G(V)$. \label{it:barOnto}
\end{enumerate}
Formulas (\ref{it:barInd}-\ref{it:barG}) are straightforward. Let us prove \eqref{it:barOnto}. Let $P \backslash G=\bigcup_{i=1}^k U_i$ be an open Nash cover such that there exist Nash sections $s_i:U_i\to G$ of the projection $pr:G\to P \backslash G$. Let $f\in \Ind_P^G(V)$. Using partition of unity on $P \backslash G$ we can assume that $f$ is supported in $pr^{-1}(U_i)$ for some $i$. Note the support of $f$ is $P$-invariant and hence it is a preimage of a compact set $S\subset U_i\subset P \backslash G$. Let $h:=f\circ s_i$. Then $h$ is compactly supported and thus $h\in \Sc(U_i)$. Thus there exists $\tilde h \in \Sc(P\times U_i,V)$ such that $h(x)=\int_{p\in P} \tilde h(p,x)\mu(p)$. Define $\tilde f \in \Sc(pr^{-1}(U_i),V)$ by
$$\tilde f(g):=\pi(g\cdot (s_i(pr(g)))^{-1})\tilde h(g\cdot (s_i(pr(g)))^{-1},pr(g)).$$
It is easy to see that $\bar \varphi (\tilde f)=f$. Indeed,
\begin{multline}
\bar \varphi (\tilde f)(x) = \int_{p\in P}\pi(p^{-1})\cdot \tilde f(px)d\mu(p)=\\=\int_{p\in P}\pi(p^{-1})\pi(pxs_i(pr(x)^{-1})\cdot \tilde h(pxs_i(pr(x))^{-1},pr(x))d\mu(p)=\\ =\pi(xs_i(pr(x)^{-1}) \int_{p\in P} \tilde h(pxs_i(pr(x))^{-1},pr(x))d\mu(p)=\\
=\pi(xs_i(pr(x)^{-1})h(pr(x))=\pi(xs_i(pr(x)^{-1})f(s_i(pr(x)))
=f(x)
\end{multline}
\end{proof}
}
\begin{proof}[Proof of \Cref{thm:compa}]
\begin{enumerate}[(i)]
\item 
\RamiC{Assume first that
$\pi=\Ind_P^G(V\otimes \psi)$ is a principal series representation. Then by \Cref{lem:Int} we have an epimorphism $\Sc(G,V)_{(\fh\times\mathfrak{p},\chi\otimes(-d \psi))}\onto \pi_{(\fh,\chi)}$.} Since $\Sc(G,V)_{(\fh\times\mathfrak{p},\chi\otimes(- d\psi))}$ is separated and finite-dimensional by Theorem \ref{thm:main}, so is $\pi_{(\fh,\chi)}$.

Now,
note that if $0\to \sigma \to \pi \to \tau\to 0$ is an exact sequence of representations in $\cM(G)$, and $\sigma_{(\fh,\chi)}$ and $\tau_{(\fh,\chi)}$
are separated and finite-dimensional, then so is $\pi_{(\fh,\chi)}$. Conversely, if $\pi_{(\fh,\chi)}$
are separated and finite-dimensional then $\tau_{(\fh,\chi)}$
is separated and finite-dimensional. Since any representation in $\cM(G)$ has finite length, and any irreducible representation is a quotient of a principal series representation by the
 Casselman subrepresentation theorem, we get that $\pi_{(\fh,\chi)}$ is separated and finite-dimensional for any $\pi\in \cM(G)$.

\item The image of $\Phi_0$ is a dense subspace since $\pi^{\HC}$ is dense in $\pi$. Thus, by (i),  $\Phi_0$ is an epimorphism.

\item The inclusion $i_0$ is given by restriction of the equivariant functional on $\pi^{\HC}$. It is indeed an inclusion since $\pi^{\HC}$ is dense. Since $\pi_{(\fh,\chi)}$ is separated and finite-dimensional, the natural pairing between it and $(\pi')^{(\fh,\chi)}$ is a perfect pairing and thus $\dim \pi_{(\fh,\chi)} = \dim (\pi')^{(\fh,\chi)}$. Similarly,
 the dual space to $ \pi^{\HC}_{(\fh,\chi)}$ is  $((\pi^{\HC})^{*})^{(\fh,\chi)}$, and thus either they are of the same finite dimension or they are both infinite-dimensional. Thus $\Phi_0$ is an isomorphism if and only if $i_0$ is.
\end{enumerate}
\end{proof}

In fact, the space $ \pi^{\HC}_{(\fh,\chi)}$ is always finite dimensional, see \cite{KO}. In the next subsection we show that the results of \cite{KS2} imply that if $K$ and $H$ are in general position then $\oH_p(\fh, \pi^{\HC})$ is finite dimensional for all $p$.

\subsection{Finiteness of higher homologies of Harish-Chandra modules}
Let $P<G$ be a minimal
parabolic subgroup such that $PH$ is open in $G$ and let $P=MAN$ be a Langlands decomposition for $P$. We will assume in the sequel that
\begin{equation}\label{=gen}
\fg = \fh +\fa +\fk.
\end{equation}
\RamiC{Note that for any real spherical $H\subset G$ there exists a maximal compact $K\subset G$ satisfying this condition, see \cite[\S 5]{KS2}.}

\par We denote by $\cU(\fg)$ the universal enveloping algebra of $\fg$ and
by $\cZ(\fg)$ its center. The following is a generalization of the Casselman-Osborne lemma for spherical subalgebras

\begin{lem}[{\cite[Lemma 5.5]{KS2}}]\label{Cas-Os} \DimaB{If \eqref{=gen} holds then}
there exists a
finite subset $\cY\subset \cU(\fg)$ such that
$$ \cU(\fg)= \cU(\fh) \cY \cZ(\fg) \cU (\fk)\, .$$
\end{lem}


\begin{prop} \label{HC} Let $V$ be a Harish-Chandra module for $(\fg, K)$.
\DimaB{If \eqref{=gen} holds then}
\begin{enumerate}
\item $V$ is finitely generated as an $\fh$-module.
\item $\oH_p(\fh, V)$ is finite dimensional for all $p$.
\end{enumerate}
\end{prop}

\begin{proof} We recall that every Harish-Chandra module is $\cZ(\fg)$-finite, i.e.
$V$ is annihilated by an ideal $\cI\triangleleft \cZ(\fg)$ of finite co-dimension.
As $V$ is finitely generated as an $(\fg, K)$-module there is a finite dimensional
$K$-invariant subspace $W\subset V$ which generates $V$ as $\fg$-module. In view of
Lemma \ref{Cas-Os},  this implies (1).

To prove (2) note that $\cU(\fh)$ is Noetherian and thus (1) implies that $V$ has a projective resolution consisting of finitely generated modules. Applying the functor of $\fh$-coinvariants to this resolution, we obtain a complex of finite dimensional vector spaces. The homologies of this complex are isomorphic to $\oH_p(\fh, V)$.
\end{proof}

\end{document}


\author{Avraham Aizenbud}
\address{Faculty of Mathematics and Computer Science, Weizmann
Institute of Science, POB 26, Rehovot 76100, Israel }
\email{aizenr@gmail.com}

\author{Dmitry Gourevitch}
 \address{Faculty of Mathematics and Computer Science, Weizmann
 Institute of Science, POB 26, Rehovot 76100, Israel }
\email{dimagur@weizmann.ac.il}

\author{Bernhard Kr\"{o}tz}
\address{Institut f\"{u}r Mathematik, Universit\"{a}t Paderborn, Warburger
Str. 100, 33098 Paderborn, Germany}

\email{bkroetz@gmx.de}

\author{Gang Liu}
\address{Institut \'{E}lie Cartan de Lorraine, Universit\'{e} de Lorraine, Ile du Saulcy, 57045 Metz,
France}
\email{gang.liu@univ-lorraine.fr}


\date{\today}
\title[Hausdorffness for Lie algebra homology of Schwartz spaces]{Erratum to ``Hausdorffness for Lie algebra homology of Schwartz spaces and applications to the comparison conjecture"}
\maketitle

The following changes to the main results of \cite{AGKL} are necessary:

\begin{enumerate}
\item
In Theorem A and Corollary B  the following  assumption is required:  the number of orbits of the complexification $H_{\bC}$ on  $G_{\bC}/P_{\bC}$ is finite, where $P$ is a minimal parabolic subgroup of $G$.

\item In Theorem C  the following  additional assumption is required: the number of orbits of $H_{\bC}$ on $X_{\bC}$ is finite.
\end{enumerate}
Presently we do not know whether these results hold  without the additional assumptions.

\bigskip \par The source of the mistake is in \cite{AGM1}, where the expressions ``real algebraic groups" and ``real algebraic manifolds" are
ambiguous. Moreover, a mistake in \cite[Definition 1.1.1]{AGM1} hints to a wrong resolution of this ambiguity, in particular in \cite[Theorem D]{AGM1}.  This result entered in the proof of Lemma 3.2.1 of \cite{AGKL}.
\par In \cite{AGM1} the terms ``real algebraic groups" and ``real algebraic manifolds" sometimes mean algebraic groups
and manifolds defined over $\R$, and sometimes real points of such. Those two meanings are not equivalent, in particular the statement that an  algebraic group $G$ defined over $\R$ acts on an algebraic manifold $X$ with finitely many orbits implies the  statement that $G(\R)$ acts on $X(\R)$ with finitely many orbits, but is not equivalent to it. Rather, it is equivalent to the stronger statement that $G(\bC)$ has finitely many orbits on $X(\bC)$.
In particular, in \cite[Theorem D]{AGM1} one needs
the stronger assumption (only then it follows from the
Bernstein-Kashiwara theorem, \cite[Thm 3.2.2]{AGM1}),  see \cite{AGM2}.
We do not know whether this theorem holds under the weaker assumption.

The argument in \cite{AGKL} proves the corrected versions of the main results (see (1,2) above), after the following revision.
\begin{enumerate}[(a)]
\item In sections 2 and 3, the expression ``real algebraic group'' has to be replaced by ``algebraic group defined over $\R$", the expression ``real algebraic manifold'' has to be replaced by ``algebraic manifold defined over $\R$".
\item One has to introduce the following notation: for an algebraic manifold $X$ defined over $\R$, a Zariski closed algebraic submanifold $Z$ and an algebraic bundle $\cE$ over $X$, denote
$\Sc_{Z}(X,\cE):=\Sc_{Z(\R)}(X(\R),\cE)$ and $\Sc_{Z}^*(X,\cE):=\Sc_{Z(\R)}^*(X(\R),\cE)$, and similarly for the special cases $\Sc(X), \Sc(X,\cE), \Sc_Z(X),$ and their dual spaces.

\item In the proof of Lemma 3.2.1, one has to add that the reason that Proposition 3.1.1 implies the finiteness of the dimension of $$\oH_0(\fh, \Sc_Z(X,\cE)/\Sc_Z(X,\cE)^i\otimes \chi)$$ is that $Z(\R)$ is a finite union of $H(\R)$-orbits.

\item In section 4, $G$ should be the group of real points of an algebraic reductive group $\bf G$ defined over $\R$, and $H$ should be  the group of real points of an algebraic subgroup $\bf H\subset \bf G$. Also, each time that we require  $H$ to be a real spherical subgroup we actually need to require the stronger condition that $\bf H$ has finitely many orbits on $\bf G/\bf P$, where $\bf P$ is a minimal parabolic subgroup of $\bf G$ defined over $\R$.
\end{enumerate}